\documentclass[a4paper,12pt]{article}

\usepackage[left=2cm,right=2cm, top=2cm,bottom=2cm,bindingoffset=0cm]{geometry}

\usepackage{verbatim}
\usepackage{amsmath}
\usepackage{amsthm}
\usepackage{amssymb}
\usepackage{delarray}
\usepackage{cite}
\usepackage{mathrsfs}
\usepackage{tikz}
\usetikzlibrary{patterns}
\usepackage{caption}
\DeclareCaptionLabelSeparator{dot}{. }
\newcommand{\al}{\alpha}
\newcommand{\be}{\beta}
\newcommand{\ga}{\gamma}

\newcommand{\la}{\lambda}
\newcommand{\om}{\omega}

\newcommand{\eps}{\varepsilon}

\newcommand{\iy}{\infty}
\theoremstyle{plain}

\newtheorem{thm}{Theorem}
\newtheorem{lem}{Lemma}

\theoremstyle{definition}

\newtheorem{alg}{Algorithm}
\newtheorem{ip}{Inverse Problem}

\theoremstyle{remark}

\newtheorem{remark}{Remark}

 \allowdisplaybreaks

\begin{document}

\begin{center}
{\large\bf Partial inverse problems for quadratic differential pencils on a graph with a loop
}
\\[0.2cm]
{\bf Natalia P. Bondarenko, Chung-Tsun Shieh} \\[0.2cm]
\end{center}

\vspace{0.5cm}

{\bf Abstract.} In this paper, partial inverse problems for the quadratic pencil of Sturm-Liouville operators on a graph with a loop were studied.  
These problems consist in recovering the pencil coefficients on one edge of the graph
(a boundary edge or the loop) from spectral characteristics, while the coefficients on the other edges are known a priori.
We obtain uniqueness theorems and constructive solutions for the partial inverse problems.

\medskip


\medskip


\vspace{1cm}

{\large \bf 1. Introduction}

\bigskip

This paper concerns inverse spectral problems for differential operators on geometrical graphs.
Such operators model wave propagation in thin structures and are also called {\it quantum graphs.}
Differential operators on graphs have applications in organic chemistry, nanotechnology, waveguide theory, mesoscopic physics, electronics and other fields
of science and engineering (see \cite{BCFK06, Exner08}). 

Inverse problems of spectral analysis consist in recovering operators from their spectral characteristics.
A good survey of inverse problems for differential operators on graphs is provided in \cite{Yur16}.
Differential {\it pencils} contain a nonlinear dependence on the spectral parameter, and therefore inverse problems
for them are more difficult for investigation than inverse problems for Sturm-Liouville operators. 
Inverse problems for quadratic  pencils on a {\it finite interval} were studied in 
\cite{GG81, BY06, BY12, HP12, Pro14}, however, some questions still remain open.
In recent years, inverse spectral problems for differential pencils on {\it graphs} have attracted much attention
\cite{Yur08-tree, Yur14-cycle, Yur16-rooted, Yur17-a, Yur17-bush, Yur19, Bond17, Bond19}.

In general, recovering a differential operator or a pencil on a graph requires much spectral data
(for example, several spectra, corresponding to different boundary conditions, see \cite{Yur08-tree}).
In this paper, we focus on {\it partial} inverse problems  which consist in determination of the pencil coefficients
on a part of the graph  while the coefficients on the other part are known a priori.
Such problems generalize the famous Hochstadt-Lieberman problem \cite{HL78} on a finite interval 
(see also \cite{GS00, Sakh01, HM04-half, MP10}).
Partial inverse problems usually require less spectral data.
%than complete ones.
In particular, it has been shown
in \cite{Yang10, Yang11, Bond18-amp, Bond18-mixed} that one needs only a fractional part of a spectrum to reconstruct the Sturm-Liouville potential
on one edge or on a part of an edge of the star-shaped graph with equal edge lengths  if the potentials are given on all the other edges.
Some more general partial inverse problems for differential operators on graphs were studied in \cite{BSh17, YW17, Bond18, BY18}.
For partial inverse problems for differential pencils on graphs, the papers \cite{Bond17, Bond19} are the only literature we know. 

In this paper, we investigate partial inverse problems for the quadratic Sturm-Liouville pencil on a graph with a loop. The purpose of this research is to 
recover pencil coefficients on a boundary edge or on the loop by using a part of the spectrum.  At first we study a sufficient condition for a subspectrum, which can 
uniquely reconstruct the pencil, in terms of eigenvalue asymptotics. Secondly, we present 
the main results of the paper  those are uniqueness theorems and constructive algorithms for solving partial inverse problems.
Although  the graph of a particular structure is studied (see Figure~\ref{fig:1}) in this paper, we note that our results could be generalized
for differential pencils on arbitrary graphs with cycles.  We choose  the graph with one loop to illustrate our  main ideas 
in order to simplify the asymptotic analysis of the spectrum, since for pencils on graphs with cycles 
this analysis is rather complicated.  Our approach to partial inverse problems is based on construction of special 
functional systems in a suitable Hilbert space, and investigation of the completeness and the Riesz-basis property of these systems.

\bigskip	

{\large \bf 2. Main Results}

\bigskip

Consider a compact graph $G$ with the vertices $\{ v_j \}_{j = 1}^m$
and the edges $\{ e_j \}_{j = 1}^m$, where $e_j = (v_j, v_m)$, $j = \overline{1, m}$,
so $e_m$ is a loop containing the only vertex $v_m$ (see Figure~\ref{fig:1}). 
Thus $\{ v_j \}_{j = 1}^{m-1}$ are boundary vertices and $v_m$ is the internal vertex.
Let all the edges $\{ e_j \}_{j = 1}^m$ have equal lengths $\pi$.  
For each edge $e_j$, we introduce a parameter $x_j \in [0, \pi]$
in such a way, that for $j = \overline{1, m}$, the end $x_j = 0$ corresponds to the vertex $v_j$, 
and the end $x_j = \pi$ corresponds to $v_m$. For $j = m$, the both ends correspond 
to the vertex~$v_m$. A function on the graph $G$ is a vector function $y = [y_j]_{j = 1}^m$,
$y_j = y_j(x_j)$, $x_j \in (0, \pi)$.

\begin{figure}[h!]
\centering
\begin{tikzpicture}
\filldraw (0, 0) circle (2pt) node[anchor=east]{$v_m$};
\filldraw (3, 1.5) circle (2pt) node[anchor=west]{$v_1$};
\filldraw (3.5, 0.5) circle (2pt) node[anchor=west]{$v_2$};
\filldraw (3, -1.5) circle (2pt) node[anchor=west]{$v_{m-1}$};
\draw (0, 0) edge node[auto]{$e_1$} (3, 1.5);
\draw (0, 0) edge node[auto]{$e_2$} (3.5, 0.5);
\draw (0, 0) edge node[anchor=north]{$e_{m-1}$} (3, -1.5);
\filldraw (3.5, -0.5) circle (0.5pt);
\filldraw (3.5, -0.2) circle (0.5pt);
\filldraw (3.4, -0.8) circle (0.5pt);
\draw (-1, 0) circle (1);
\draw (-2.3, 0) node{$e_m$};
\end{tikzpicture}
\caption{Graph $G$}
\label{fig:1}
\end{figure}

Consider the following boundary value problem $L$ on the graph $G$:
\begin{align} \label{eqv}
& -y_j''(x_j) + q_j(x_j) y_j(x_j) + 2 \la p_j(x_j) y_j(x_j) = \la^2 y_j(x_j),  
\quad x_j \in (0, \pi), \quad j = \overline{1, m}, \\ \label{bc}
& y_j(0) = 0, \quad j = \overline{1, m-1}, \\ \label{cont}
& y_m(0) = y_j(\pi), \quad j = \overline{1, m}, \\ \label{kirch}
& y_m'(0) = \sum_{j = 1}^m y_j'(\pi).
\end{align}
Here \eqref{eqv}--\eqref{kirch} is a differential pencil with the nonlinear dependence on the spectral parameter $\la$;
$y = [y_j]_{j = 1}^m$, $p = [p_j]_{j = 1}^m$ and $q = [q_j]_{j = 1}^m$ are complex-valued functions on $G$, $y_j \in W_1^2[0, \pi]$, $p_j \in AC[0, \pi]$
(the class of absolutely continuous functions) and $q_j \in L(0, \pi)$, $j = \overline{1, m}$. We have the Dirichlet boundary conditions
\eqref{bc} in the boundary vertices and the standard matching conditions \eqref{cont}-\eqref{kirch} in the internal vertex.
The relations \eqref{cont} are continuity conditions on the function $y$, and the relation \eqref{kirch} is called the Kirchhoff's
condition, since it expresses Kirchhoff's law in electrical circuits. One can also study boundary and matching conditions
in a more general form (see \cite{Yur14-cycle}).

Define $\al_j = \frac{1}{\pi} \int_0^{\pi} p_j(t) \, dt$. Below we will say that assumption $(A)$ holds,
if the following conditions $(i)$--$(iii)$ are fulfilled:

$(i)$ $\al_j \in \mathbb R$, $j = \overline{1, m}$,

\smallskip

$(ii)$ $\al_j \not \equiv \al_k \pmod 1$, $j, k = \overline{1, m}$, $j \ne k$.

\smallskip

$(iii)$ $\al_m = 0$.

\smallskip

Here and below the notation $a \equiv b \pmod c$ means that $\frac{a - b}{c} \in \mathbb Z$.

Our first result is the following theorem, which describes asymptotic behavior of the eigenvalues of $L$.

\begin{thm} \label{thm:asympt}
Suppose that assumption $(A)$ holds.
Then the boundary value problem $L$ has a countable set of eigenvalues, which can be numbered as 
$\{ \la_{nk} \}_{(n, k) \in \mathcal I}$, counting with their multiplicities, and satisfy 
the following asymptotic  behavior 
\begin{equation} \label{asymptla}
\la_{nk} = 2 n + \be_k + O\left(|n|^{-1}\right), \quad |n| \to \iy, \quad k = \overline{1, 2m},
\end{equation}
where $\be_k \in (-1, 1)$ are distinct numbers for $k = \overline{1, 2m-1}$, $\be_{2m} = 0$, $\be_k \not \equiv \al_j \pmod 1$, $k = \overline{1, 2m}$, $j = \overline{1, m-1}$,
and 
$$
\mathcal I := \{ (n, k) \colon n \in \mathbb Z, \, k = \overline{1, m+1}\} \cup 
\{ (n, k) \colon n \in \mathbb Z \backslash \{ 0 \}, \, k = \overline{m + 2, 2 m} \}.
$$
\end{thm}

We will prove Theorem~\ref{thm:asympt} in Section~3.
Note that even if the vector-functions $p$ and $q$ are real-valued, the problem $L$ may have multiple and/or complex eigenvalues.
Condition $(iii)$ is nonrestrictive, since it can be easily achieved by a shift of the problem  as following:
$\la \mapsto \la + C$, $p_j \mapsto p_j + C$, $q_j \mapsto q_j - 2 C p_j - C^2$, $j = \overline{1, m}$, $C = -\al_m$.
Assumption $(ii)$ is imposed for simplicity. The case, when $(ii)$ is violated, requires some technical modifications. 

Let us introduce the characteristic function of the problem $L$.
For each fixed $j = \overline{1, m}$, let $S_j(x_j, \la)$ and $C_j(x_j, \la)$ be the solutions of equation \eqref{eqv}
under the initial conditions
\begin{equation} \label{ic}
    S_j(0, \la) = C_j'(0, \la) = 0, \quad S_j'(0, \la) = C_j(0, \la) = 1.
\end{equation}
Define the function $d_m(\la) := S_m'(\pi, \la) + C_m(\pi, \la) - 2$.
Using the standard methods (see, for example, \cite{FY01, BSh17}), one can show that
the eigenvalues of the boundary value problem $L$ coincide with the zeros of the analytic characteristic function
\begin{equation} \label{defDelta}
   \Delta(\la) := \sum_{j = 1}^{m-1} S_j'(\pi, \la) \prod_{\substack{k = 1 \\ k \ne j}}^m S_k(\pi, \la) + 
	d_m(\la) \prod_{k = 1}^{m-1} S_k(\pi, \la),
\end{equation}
counting with their multiplicities.

Let $\Lambda = \{ \la_{\theta} \}_{\theta \in \Theta}$ be a subset  of the spectrum of the problem $L$, possibly containing multiple values.
Here $\Theta$ is some countable set of indices. For example, one can consider $\Theta = \mathcal N$, where $\mathcal N$ is defined in \eqref{defLa}. 
Denote  $m_{\theta}$ the multiplicity of the value $\la_{\theta}$ in  $\Lambda$, i.e.
$m_{\theta}$ equals to the number of such indices $\tau \in \Theta$ that $\la_{\tau} = \la_{\theta}$. We assume that $m_{\theta}$
is less or equal to the multiplicity of the eigenvalue $\la_{\theta}$ in the spectrum.

Suppose $\Lambda = \{ \la_{\theta} \}_{\theta \in \Theta}$ satisfies the following assumption $(B)$.

\smallskip

$(B)$ The set $\Theta$ can be divided into two subsets: $\Theta = \Theta_1 \cup \Theta_2$
with the following properties. 

\smallskip 

(i) For each $\theta \in \Theta_1$, we have $S_j(\pi, \la_{\theta}) \ne 0$, $j = \overline{2, m}$.

\smallskip

(ii) For each  $\theta \in \Theta_2$ there exists an index $j_{\theta} \in \{ 2, \dots, m \}$, such that
$S_{j_{\theta}}(\pi, \la_{\theta}) = 0$ and $S_j(\pi, \la_{\theta}) \ne 0$ for all $j = \overline{2, m} \backslash \{ j_{\theta} \}$.
If $j_{\theta} = m$, then additionally $d_m(\la_{\theta}) \ne 0$.

\smallskip

Note that the values $S_j(\pi, \la)$ and $S_j'(\pi, \la)$ can not be simultaneously equal zero for every $\la$,
and $\Delta(\la_{\theta}) = 0$ for $\theta \in \Theta$.
Consequently, in view of \eqref{defDelta}, 
we have $S_1(\pi, \la_{\theta}) \ne 0$ for $\theta \in \Theta_1$ and
$S_1(\pi, \la_{\theta}) = 0$ for $\theta \in \Theta_2$.

We consider the following partial inverse problem on the graph $G$.

\begin{ip} \label{ip:1}
Given the functions $\{ p_j \}_{j = 2}^m$, $\{ q_j \}_{j = 2}^m$ and a subspectrum $\Lambda$, satisfying $(B)$,
find $p_1$ and $q_1$.
\end{ip}

Inverse Problem~\ref{ip:1} consists in recovering the pencil coefficients on a boundary edge.
Below we  provide conditions on the choice of a subspectrum $\Lambda$,  which is sufficient for Inverse Problem~\ref{ip:1} to be uniquely
solvable. Moreover we shall show a   constructive solution. 
Let
\begin{equation} \label{defLa}
   \Lambda = \{ \la_{nk} \}_{(n, k) \in \mathcal N}, \quad
   \mathcal N := \{(n, 1) \colon n \in \mathbb Z \backslash \{ 0 \} \} \cup
    \{ (n, k) \colon n \in \mathbb Z, \, k = \overline{2, 4} \},
\end{equation}
where the numbers $\la_{nk}$ obey the asymptotic relation \eqref{asymptla} with 
$\be_k$ satisfying the claim of Theorem~\ref{thm:asympt}.
Note that the numeration of the eigenvalues in Theorem~\ref{thm:asympt} is not unique. Any finite number of 
the first eigenvalues in the set $\{ \la_{nk} \}$ can stay in an arbitrary order, and the order of the numbers 
$\{ \be_k \}_{k =1}^{2m}$ is also not fixed. So, roughly speaking, we choose any four subsequences 
$\{ \la_{nk} \}_{n \in \mathbb Z}$, satisfying \eqref{asymptla} with distinct $\be_k$, and exclude one value.

Along with $L$ we consider another boundary value problem $\tilde L$ of the same form \eqref{eqv}-\eqref{kirch},
but with different coefficients $\tilde p = [\tilde p_j]_{j = 1}^m$ and $\tilde q = [\tilde q_j]_{j= 1}^m$. 
We agree that if a symbol $\ga$ denotes an object, related to $L$,
the symbol $\tilde \ga$ with tilde denotes the analogous object, related to $\tilde L$. Now we are ready
to formulate the uniqueness theorem for the solution of Inverse Problem~\ref{ip:1}.

\begin{thm} \label{thm:uniq1}
Suppose that for $j = \overline{2, m}$, we have $p_j = \tilde p_j$ in $AC[0,\pi]$ and $q_j = \tilde q_j$ 
in $L$. Denote by $\Lambda$ a subspectrum of $L$ and by $\tilde \Lambda$  a subspectrum of $\tilde L$. 
Suppose $\Lambda $ is of the form \eqref{defLa}, $\Lambda = \tilde \Lambda$, and assumptions $(A)$ and $(B)$ hold for the both pairs $(L, \Lambda)$ and $(\tilde L, \tilde \Lambda)$. 
Then $p_1 = \tilde p_1$ in $AC[0, \pi]$ and $q_1 = \tilde q_1$ in $L(0, \pi)$. Thus, under the above assumptions, the solution of Inverse Problem~\ref{ip:1}
is unique.
\end{thm}

We shall  prove Theorem~\ref{thm:uniq1} in Section~4 and develop a constructive algorithm for solving Inverse Problem~\ref{ip:1}
in Section~5.

We are also interested in recovering the coefficients $p_m$ and $q_m$ of the considered pencil on the loop $e_m$.
For that purpose, together with a subspectrum, we need some additional data associated with the periodic inverse 
problem for the pencil on the loop. That problem has been studied in  \cite{Yur12-period}. Let us briefly describe
its results.
Let $\{ \nu_n \}_{n \in \mathbb Z}$ be the zeros of the analytic function $S_m(\pi, \la)$.
Set $Q(\la) := C_m(\pi, \la) - S_m'(\pi, \la)$ and
$$
    \om_n := \left\{\begin{array}{ll}
		0, & \quad Q(\nu_n) = 0, \\
		+1, & \quad Q(\nu_n) \ne 0, \: \arg Q(\nu_n) \in [0, \pi), \\
		-1, & \quad Q(\nu_n) \ne 0, \: \arg Q(\nu_n) \in [\pi, 2\pi).
	     \end{array}\right. 		
$$
The following condition   

\smallskip

$(C)$ $\om_n \ne 0$ for all $n \in \mathbb Z,$

\smallskip
\noindent will play a crucial role in the remaining of this section. At first we have 

\begin{lem} \label{lem:om}
Assumption $(C)$ implies that the functions $S_m(\pi, \la)$ and $d_m(\la)$ do not have common zeros.
\end{lem}

This lemma will be proved in Section 6.

Put $\Omega := \{ \om_n \}_{n \in \mathbb Z}$.
The following inverse problem has been studied in \cite{Yur12-period}.

\begin{ip} \label{ip:period}
Given $d_m(\la)$, $S_m(\la)$ and $\Omega$, construct $p_m$ and $q_m$.
\end{ip}

The results of \cite{Yur12-period} imply the following uniqueness theorem and a constructive algorithm
for solving Inverse Problem~\ref{ip:period}.

\begin{thm} \label{thm:period}
Suppose condition  (C) holds, then  the specification of $d_m(\la)$, $S_m(\pi, \la)$ and $\Omega$ uniquely determines functions $p_m$ and $q_m.$
\end{thm}

Let us return to our  problem.  We impose the following assumption $(D)$ on the subspectrum
$\Lambda = \{ \la_{\theta} \}_{\theta \in \Theta}$.

\smallskip

$(D)$ The set $\Theta$ can be divided into two subsets: $\Theta = \Theta_1 \cup \Theta_2$
with the following properties. 

\smallskip 

(i) For each $\theta \in \Theta_1$, we have $S_j(\pi, \la_{\theta}) \ne 0$, $j = \overline{1, m-1}$.

\smallskip

(ii) For each $\theta \in \Theta_2$ there exists an index $j_{\theta} \in \{ 1, \dots, m -1 \}$, such that
$S_{j_{\theta}}(\pi, \la_{\theta}) = 0$ and $S_j(\pi, \la_{\theta}) \ne 0$ for all $j = \overline{1, m-1} \backslash \{ j_{\theta} \}.$

\smallskip

Recall that $\Delta(\la_{\theta}) = 0$ for $\theta \in \Theta$, and $S_j(\pi, \la)$ do not have common zeros with $S_j'(\pi, \la)$ for all $j = \overline{1, m}$.
Consequently, in view of \eqref{defDelta}, assumption $(C)$ and Lemma~\ref{lem:om}, 
we have $S_m(\pi, \la_{\theta}) \ne 0$ for $\theta \in \Theta_1$ and $S_m(\pi, \la_{\theta}) = 0$ for $\theta \in \Theta_2$.

Next, we consider the following partial inverse problem of recovering the coefficients of the pencil on the loop.

\begin{ip} \label{ip:m}
Let the functions $\{ p_j \}_{j = 1}^{m-1}$, $\{ q_j \}_{j = 1}^{m-1}$,  $\Omega$ and a subspectrum $\Lambda$ 
 which satisfies  assumptions $(C)$ and $(D)$ be given. Find $p_m$ and $q_m$.
\end{ip}

In Section~5, we will reduce Inverse Problem~\ref{ip:m} to the periodic Inverse Problem~\ref{ip:period} 
and prove the following uniqueness theorem.

\begin{thm} \label{thm:uniqm}
Suppose  we have $p_j = \tilde p_j$ in $AC[0,\pi]$ and $q_j = \tilde q_j$ 
in $L(0, \pi)$ for $j = \overline{1, m-1}$,  the subspectrum $\Lambda$ is of the form \eqref{defLa}, $\Lambda = \tilde \Lambda$, $\Omega = \tilde \Omega$
and assumptions $(A)$, $(C)$ and $(D)$  are fulfilled   for  $(L, \Lambda)$ and  $(\tilde L, \tilde \Lambda),$
then $p_m = \tilde p_m$ in $AC[0, \pi]$
and $q_m = \tilde q_m$ in $L(0, \pi)$. Thus, under the above assumptions, the solution of Inverse Problem~\ref{ip:m}
is unique.
\end{thm}

\bigskip

{\large \bf 3. Asymptotic Behavior of the Eigenvalues}

\bigskip

This section is devoted to the proof of Theorem~\ref{thm:asympt}. 
Using the transformation operators \cite{GG81, Pro13}, one can obtain the relations
\begin{align}  \label{intS}
S_j(\pi, \la) & = \frac{\sin(\la - \al_j)\pi}{\la} + \frac{1}{\la} \int_{-\pi}^{\pi} K_j(t) \exp(i \la t) \, dt, \\ \label{intSp}
S_j'(\pi, \la) & = \cos (\la - \al_j) \pi + \int_{-\pi}^{\pi} N_j(t) \exp(i \la t) \, dt, \\ \label{intC}
C_j(\pi, \la) & = \cos (\la - \al_j) \pi + \int_{-\pi}^{\pi} L_j(t) \exp(i \la t) \, dt, 
\end{align} 
where $K_j$, $N_j$ and $L_j$ are some  continuous functions on $[-\pi, \pi]$, $j = \overline{1, m}$.
In particular, relations \eqref{intS}-\eqref{intC} yield the asymptotic formulas
\begin{align} \label{asymptS} 
S_j(\pi, \la) & = \frac{\sin(\la - \al_j)\pi}{\la} + O\left( |\la|^{-1} \exp(\pi|\mbox{Im}\,\la|) \right), \\ \label{asymptSp}
S_j'(\pi, \la) & = \cos (\la - \al_j) \pi + O\left( \exp(\pi|\mbox{Im}\,\la|) \right), \\ \nonumber
C_j(\pi, \la) & = \cos (\la - \al_j) \pi + O\left(\exp(\pi|\mbox{Im}\,\la|) \right), 
\end{align} 
as $|\la| \to \iy$. Substituting the formulas \eqref{asymptS} and \eqref{asymptSp} into \eqref{defDelta}, we get 
\begin{equation} \label{asymptDelta}
    \Delta(\la) = \la^{1-m} \Big(d(\la) + O\left(|\la|^{-1} \exp(m \pi |\mbox{Im}\,\la|)\right)\Big), \quad |\la| \to \iy,
\end{equation} 
where
\begin{equation} \label{defd}
  d(\la) = \sum_{j = 1}^{m-1} \cos (\la - \al_j) \pi \left( \prod_{\substack{k = 1 \\ k \ne j}}^{m-1} \sin (\la - \al_k) \pi \right) \sin \la \pi
+ 2 (\cos \la \pi - 1) \prod_{k = 1}^{m-1} \sin (\la - \al_k) \pi.
\end{equation}

\begin{lem} \label{lem:d1}
Under assumption $(A)$, the set of zeros of the function $d(\la)$ is described as follows:
\begin{equation} \label{zerosd}
 2 n + \be_k, \quad n \in \mathbb Z, \quad k = \overline{1, 2m},
\end{equation}
where $\be_k \in (-1, 1)$ are distinct real numbers for $k = \overline{1, 2m-1}$, and $\be_{2m} = 0$.
\end{lem}

\begin{proof}
Clearly, the function $d(\la)$ is $2$-periodic, so it is sufficient to study its zeros in $[-1, 1)$. Without loss of generality,
we assume that $0 < \al_1 < \al_2 < \dots < \al_{m-1} < 1$. Note that $d(\al_j) \ne 0$  for $j = \overline{1, m-1}$, and $d(0) = 0$. Hence
the zeros of $d(\la)$ in $[-1, 1)$, except $\la = 0$, coincide with the roots of the equation
$\kappa_1(\la) = \kappa_2(\la)$, where
\begin{equation*}
   \kappa_1(\la) := \sum_{j = 1}^{m-1} \cot (\la - \al_j) \pi, \quad
\kappa_2(\la) = \frac{2(1 - \cos \la \pi)}{\sin \la \pi}.
\end{equation*}
One can easily check that the function $\kappa_1(\la)$ is strictly decreasing on the intervals
\begin{equation} \label{intervals}
   (-1, \al_1 - 1), \quad (\al_j - 1, \al_{j+1}-1), \quad (\al_{m-1}-1, \al_1), \quad (\al_j, \al_{j+1}), \quad 
  (\al_{m-1}, 1), \quad j = \overline{1, m-2},
\end{equation}
and
$$
    \kappa_1(\al_j - 1 - 0) = \kappa_1(\al_j - 0) = -\iy, \quad \kappa_1(\al_j - 1 + 0) = \kappa_1(\al_j + 0) = +\iy;
$$
the function $\kappa_2(\la)$ strictly increases on $(-1, 1)$, $\kappa_2(-1 + 0) = -\iy$ and $\kappa_2(1 - 0) = +\iy$.
Consequently, the equation $\kappa_1(\la) = \kappa_2(\la)$ has exactly one root $\be_k$ in each of the $2m-1$ intervals \eqref{intervals}.
Adding the zero $\la = 0$ and using $2$-periodicity of $d(\la)$, we show that the function $d(\la)$ has zeros in the form \eqref{zerosd}.
Let us prove that there exist no other zeros.

Define the function
\begin{equation} \label{defp}
   p(\la) = \prod_{k = 1}^{2 m} \sin\left( (\la - \be_k) \tfrac{\pi}{2} \right).
\end{equation}
Clearly, the function $\dfrac{d(\la)}{p(\la)}$ is entire. Using \eqref{defd} and \eqref{defp}, we obtain the estimates
$$
    d(\la) = O(\exp(m \pi |\mbox{Im}\,\la|)), 
$$
$$
    |p(\la)| \ge C \exp(m \pi|\mbox{Im}\,\la|), \quad |\la| > \la^*, \quad \eps < |\arg \la| < \pi - \eps,
$$
for some positive constants $\la^*$ and $\eps$.
Consequently, the entire function $\dfrac{d(\la)}{p(\la)}$ is bounded in the sectors 
$\{ \la \in \mathbb C \colon \eps < |\arg \la| < \pi - \eps \}$.
Here and below we use the symbol $C$ for various constants, not depending on $\la$.
Applying Phragmen-Lindel\"of's and Liouville's theorems (see \cite{BFY}), we conclude that $\dfrac{d(\la)}{p(\la)}$ is constant.
So the function $d(\la)$ does not have other zeros except the zeros of $p(\la)$ which are described by \eqref{zerosd}.
\end{proof}

Using Lemma~\ref{lem:d1},  formula  \eqref{asymptDelta} and 
Rouch\'e's theorem (see \cite[Theorem~1.1.3]{FY01}),  on can   obtain the assertion of Theorem~\ref{thm:asympt}.

\bigskip

{\large \bf 4. Uniqueness Theorem for Inverse Problem~\ref{ip:1}}

\bigskip

In this section, we shall prove Theorem~\ref{thm:uniq1}. Rewrite the relation \eqref{defDelta} in the form
\begin{equation} \label{Delta1}
    \Delta(\la) = A_1(\la) S_1(\pi, \la) + B_1(\la) S_1'(\pi, \la),
\end{equation}
where
\begin{equation} \label{defAB1}
    A_1(\la) := \sum_{j = 2}^m S_j'(\pi, \la) \prod_{\substack{k = 2 \\ k \ne j}}^m S_k(\pi, \la) + 
	d_m(\la) \prod_{k = 2}^{m-1} S_k(\pi, \la), \quad
    B_1(\la) := \prod_{k = 2}^m S_k(\pi, \la).		
\end{equation}

Substituting \eqref{intS} and \eqref{intSp} for $j = 1$ into \eqref{Delta1} and multiplying the result by $\la$,
we get
\begin{equation} \label{intAB1}
     A_1(\la) \int_{-\pi}^{\pi} K_1(t) \exp(i \la t) \, dt + \la B_1(\la) \int_{-\pi}^{\pi} N_1(t) \exp(i \la t) \, dt
     - G_1(\la) = \la \Delta(\la), 
\end{equation}
where
\begin{equation} \label{defG1}
   G_1(\la) := -A_1(\la) \sin (\la - \al_1) \pi - \la B_1(\la) \cos (\la - \al_1) \pi. 		
\end{equation}
Here $``1``$ denotes the connection with the inverse problem, that consists in recovering the pencil coefficients on
the edge $e_1$. 

Denote $\mathcal H = L_2(-\pi, \pi) \oplus L_2(-\pi, \pi)$ the complex Hilbert space. The scalar product
and the norm in $\mathcal H$ are defined as follows.
$$
    (g, h)_{\mathcal H} = \int_{-\pi}^{\pi} (\overline{g_1(t)} h_1(t) + \overline{g_2(t)} h_2(t)) \, dt, \quad
    \| g \|_{\mathcal H} = \sqrt{\int_{-\pi}^{\pi} (|g_1(t)|^2 + |g_2(t)|^2) \, dt}.
$$

Define the vector-functions
\begin{equation} \label{defh1}
    f_1(t) = \begin{bmatrix} \overline{K_1(t)} \\ \overline{N_1(t)} \end{bmatrix}, \quad 
    h_1(t, \la) = \begin{bmatrix} A_1(\la) \\ \la B_1(\la) \end{bmatrix} \exp(i \la t),
\end{equation}
where the functions $K_1$ and $N_1$ are as there in \eqref{intS} and \eqref{intSp}.
Consequently, $f_1 \in \mathcal H$. Obviously, $h_1(t, \la)$ and its derivatives with respect to $\la$ also belong to $\mathcal H$
for each fixed $\la \in \mathbb C$.
Therefore the relation \eqref{intAB1} can be rewritten in the the form
\begin{equation} \label{scal1}
    (f_1(t), h_1(t, \la))_{\mathcal H} = G_1(\la) + \la \Delta(\la).
\end{equation}

Let $\Lambda = \{ \lambda_{\theta} \}_{\theta \in \Theta}$ be a subspectrum of the problem $L$ which satisfies assumption $(B)$.
Note that $\la_{\theta}$ is a zero of the function $\Delta(\la)$ of multiplicity at least $m_{\theta}$.
For $\theta \in \Theta_1$ we have $S_1(\pi, \la_{\theta}) \ne 0$ and $B_1(\la_{\theta}) \ne 0$.
Consequently,  for $\theta \in \Theta_1,$ the relations \eqref{Delta1} and \eqref{scal1} imply 
\begin{equation} \label{passABS1}
    \frac{d^j}{d\la^j} \left( \frac{A_1(\la)}{B_1(\la)}\right) \Big|_{\la = \la_{\theta}} = -\frac{d^j}{d\la^j} 
   \left( \frac{S_1'(\pi, \la)}{S_1(\pi, \la)}\right) \Big|_{\la = \la_{\theta}}, \quad j = \overline{0, m_{\theta}-1},
\end{equation}  
\begin{equation} \label{scal2}
    \biggl(f_1(t), \frac{d^j}{d\la^j} \left(\frac{h_1(t, \la)}{B_1(\la)}\right)\biggr)_{\mathcal H} = 
    \frac{d^j}{d\la^j} \left(\frac{G_1(\la)}{B_1(\la)}\right), \quad \la = \la_{\theta}, \quad 
	j = \overline{0, m_{\theta}-1}.
\end{equation}

For $\theta \in \Theta_2$ we have $S_1(\pi, \la_{\theta}) = 0$ and $B_1(\la_{\theta}) = 0$, but
$S_1'(\pi, \la_{\theta}) \ne 0$ and $A_1(\la_{\theta}) \ne 0$. In this case, it follows from the relations \eqref{Delta1} and \eqref{scal1},
that
\begin{equation} \label{passABS2}
    \frac{d^j}{d\la^j} \left( \frac{B_1(\la)}{A_1(\la)}\right) \Big|_{\la = \la_{\theta}} = -\frac{d^j}{d\la^j} 
   \left( \frac{S_1(\pi, \la)}{S_1'(\pi, \la)}\right) \Big|_{\la = \la_{\theta}}, \quad j = \overline{0, m_{\theta}-1},
\end{equation} 
\begin{equation} \label{scal3}
    \biggl(f_1(t), \frac{d^j}{d\la^j} \left(\frac{h_1(t, \la)}{A_1(\la)}\right)\biggr)_{\mathcal H} = 
    \frac{d^j}{d\la^j} \left(\frac{G_1(\la)}{A_1(\la)}\right), \quad \la = \la_{\theta}, \quad 
	j = \overline{0, m_{\theta}-1}.
\end{equation}

For definiteness,  we assume that $0 \not \in \Lambda$ and $0 \not \in \Theta$.
Denote $\Theta_0 = \Theta \cup \{ 0 \}$, $\la_0 := 0$, $m_0 := 1$.
One can deal with the case $0 \in \Lambda$, by applying a shift of the spectrum and minor technical modifications.

Since the function $S_1(\pi, \la)$ is entire, the relation \eqref{intS} yields
\begin{equation} \label{scal4}
    \int_{-\pi}^{\pi} K_1(t) \, dt = \sin \al_1 \pi.
\end{equation}

Combining the relations \eqref{scal2}, \eqref{scal3} and \eqref{scal4} together, we arrive at the following result
\begin{equation} \label{main1}
    (f_1, h_{1,\theta j})_{\mathcal H} = G_{1, \theta j}, \quad \theta \in \Theta_0, \quad j = \overline{0, m_{\theta}-1},
\end{equation}
where
\begin{gather*}
    h_{1,00}(t) = \begin{bmatrix} 1 \\ 0 \end{bmatrix}, \quad G_{1, 00} = \sin \al_1 \pi, \\
    h_{1,\theta j}(t) = \frac{d^j}{d\la^j} \left(\frac{h_1(t, \la)}{B_1(\la)}\right) \Big|_{\la = \la_{\theta}}, \quad
    G_{1,\theta j} = \frac{d^j}{d\la^j} \left(\frac{G_1(\la)}{B_1(\la)}\right) \Big|_{\la = \la_{\theta}}, \quad
    \theta \in \Theta_1, \\
    h_{1,\theta j}(t) = \frac{d^j}{d\la^j} \left(\frac{h_1(t, \la)}{A_1(\la)}\right) \Big|_{\la = \la_{\theta}}, \quad
    G_{1,\theta j} = \frac{d^j}{d\la^j} \left(\frac{G_1(\la)}{A_1(\la)}\right) \Big|_{\la = \la_{\theta}}, \quad
    \theta \in \Theta_2.
\end{gather*}

Thus, we have derived the relations \eqref{main1} which are crucial for solving Inverse Problem~\ref{ip:1}.
The vector function $f_1$, containing $K_1$ and $N_1$, will be unknown while 
the vector functions $h_{1,\theta j}(t)$ and the right-hand sides $G_{1, \theta j}$, as we will see below, 
can be constructed by using the known data. 

Denote 
\begin{equation} \label{defHG1}
\mathscr H_1(\Lambda) := \{ h_{1,\theta j}(t) \}_{\theta \in \Theta_0, \, j = \overline{0, m_{\theta}-1}}, \quad
\mathscr G_1(\Lambda) := \{ G_{1,\theta j} \}_{\theta \in \Theta_0, \, j = \overline{0, m_{\theta}-1}}.
\end{equation}
  
Then  the left-hand side of \eqref{main1} contains Fourier coefficients
of $f_1$ with respect to the system $\mathscr H_1(\Lambda)$.

Now we are ready to prove the following abstract uniqueness theorem.

\begin{thm} \label{thm:uniq1-abs}
Suppose  $p_j = \tilde p_j$ in $AC[0,\pi]$  and $q_j = \tilde q_j$  
in $L(0, \pi)$ for $j = \overline{2, m},$ 
 $\Lambda = \tilde \Lambda$, $\al_1 = \tilde \al_1,$ assumption $(B)$ holds for the both pairs $(L, \tilde \Lambda)$
and $(\tilde L, \tilde \Lambda)$, and the system $\mathscr H_1(\Lambda)$ is complete in $\mathcal H$. Then $p_1 = \tilde p_1$ in $AC[0, \pi]$
and $q_1 = \tilde q_1$ in $L(0, \pi)$. Thus, under the above assumptions, the solution of Inverse Problem~\ref{ip:1}
is unique.
\end{thm}

In contrast to the uniqueness Theorem~\ref{thm:uniq1}, the statement of Theorem~\ref{thm:uniq1-abs} contains the completeness
condition for $\mathscr H_1(\Lambda)$ and has no requirements on asymptotic behavior of the subspectrum.

\begin{proof}[Proof of Theorem~\ref{thm:uniq1-abs}]
Consider boundary value problems $L$ and $\tilde L$ and their subspectra $\Lambda$ and $\tilde \Lambda$, respectively,
satisfying the conditions of the theorem. By the assumptions, we have $S_j(x, \la) \equiv \tilde S_j(x, \la)$ for $j = \overline{2, m}$.
In view of \eqref{defAB1}, this yields $A_1(\la) \equiv \tilde A_1(\la)$ and $B_1(\la) \equiv \tilde B_1(\la)$.
Consequently, $h_1(t, \la) \equiv \tilde h_1(t, \la)$. Taking the equality $\al_1 = \tilde \al_1$ into account, 
we also conclude that $G_1(\la) \equiv \tilde G_1(\la)$. Since $\Lambda = \tilde \Lambda$, we have 
$\mathscr H_1(\Lambda) = \tilde{\mathscr H}_1(\tilde \Lambda)$, $\mathscr G_1(\Lambda) = \tilde{\mathscr G}_1(\tilde \Lambda)$.
Hence along with the relations \eqref{main1}, we obtain
\begin{equation} \label{maint}
    (\tilde f_1, h_{1, \theta j})_{\mathcal H} = G_{1, \theta j}, \quad \theta \in \Theta_0, \quad j = \overline{0, m_{\theta}-1}.
\end{equation}
Subtracting \eqref{maint} from \eqref{main1}, we see that the difference $f_1 - \tilde f_1$ is orthogonal in $\mathcal H$
to all the elements of the system $\mathscr H_1(\Lambda)$, which is complete. Therefore $f_1 = \tilde f_1$,
so $K_1 = \tilde K_1$ and $N_1 = \tilde N_1$ in $L_2(-\pi, \pi)$. Now \eqref{intS} and \eqref{intSp} yield
$S_1(\pi, \la) \equiv \tilde S_1(\pi, \la)$, $S_1'(\pi, \la) \equiv \tilde S_1'(\pi, \la)$.

Note that $\dfrac{S_1'(\pi, \la)}{S_1(\pi, \la)}$ is the Weyl function of the boundary value problem
on a finite interval $(0, \pi)$:
\begin{equation} \label{bvp}
-y_1''(x_1) + q_1(x_1) y_1(x_1) + 2 \la p_1(x_1) y_1(x_1) = \la^2 y_1(x_1),  
\quad y_1(0) = y_1(\pi) = 0.
\end{equation}

It has been proved in \cite{BY12}, that the functions $p_1 \in AC[0, \pi]$ and $q_1 \in L(0,\pi)$ are uniquely specified 
by the Weyl function, and can be constructed by the method of spectral mappings. 
Hence $p_1 = \tilde p_1$ in $AC[0, \pi]$ and $q_1 = \tilde q_1$ in $L(0,\pi)$.
\end{proof}

Suppose that the system of vector-functions $\mathscr H_1(\Lambda)$ is an unconditional basis in $\mathcal H$,
i.e. the normalized system $\{ h_{1,\theta j}(t) / \| h_{1,\theta j}(t) \|_{\mathcal H} \}$ is a Riesz basis.
One can find more about Riesz bases in \cite[Section~1.8.5]{FY01}. 
The proof of Theorem~\ref{thm:uniq1-abs} gives the following algorithm for solving Inverse Problem~\ref{ip:1},
if we know additionally the number $\al_1$.

\begin{alg} \label{alg:1}
Suppose the functions $\{ p_j \}_{j = 2}^m$, $\{ q_j \}_{j = 2}^m$, the subspectrum $\Lambda$
and the number $\al_1$ are given, and assumption $(B)$ holds.

\begin{enumerate}
\item Find the solutions $S_j(x_j, \la)$ of equations~\eqref{eqv} for $j = \overline{2, m}$, satisfying the initial 
conditions~\eqref{ic}.
\item Construct the functions $A_1(\la)$ and $B_1(\la)$ by \eqref{defAB1}.
\item Find the functions $h_1(t, \la)$ and $G_1(\la)$ by \eqref{defh1} and \eqref{defG1}, and then
use them together with the given subspectrum $\Lambda$ to construct $\mathscr H_1(\Lambda)$ and $\mathscr G_1(\Lambda)$
(see~\eqref{defHG1}). 
\item Determine the vector-function $f_1(t)$, using its coordinates in the Riesz basis (see \eqref{main1}), 
i.e. find $K_1(t)$ and $N_1(t)$.
\item Find $S_1(\pi, \la)$ and $S_1'(\pi, \la)$ by \eqref{intS} and \eqref{intSp}.
\item Recover the coefficients $p_1$ and $q_1$ of the boundary value problem \eqref{bvp} from the Weyl function
$\dfrac{S_1'(\pi, \la)}{S_1(\pi, \la)}$, solving the inverse problem on a finite interval by the method of spectral mappings
(see \cite{BY12}).
\end{enumerate} 
\end{alg}

We will show that the assumptions in  Theorem~\ref{thm:uniq1} imply the completeness of the system $\mathscr H_1(\Lambda)$,
and then derive Theorem~\ref{thm:uniq1} as a corollary of Theorem~\ref{thm:uniq1-abs}.
First we investigate the system of vector-functions 
$\mathscr V(\Lambda) := \{ v_{\theta j}(t) \}_{\theta \in \Theta_0, \, j = \overline{0, m_{\theta} - 1}}$, where
\begin{equation} \label{defV}
v_{\theta j}(t) = \frac{d^j}{d\la^j} v(t, \la) \Big|_{\la = \la_{\theta}}, \quad
v(t, \la) = \begin{bmatrix} S_1'(\pi, \la) \\ -\la S_1(\pi, \la) \end{bmatrix} \exp(i \la t).
\end{equation} 

\begin{lem} \label{lem:HV}
Denote by $\Lambda$ a subspectrum  of the boundary value problem $L$  which satisfies  assumption $(B)$.
Then the system $\mathscr H_1(\Lambda)$ is complete in $\mathcal H$ if and only if so 
is the system $\mathscr V(\Lambda)$.
\end{lem}

\begin{proof}
Consider the relations
\begin{equation} \label{scalw}
    (w, h_{1,\theta j})_{\mathcal H} = 0, \quad \theta \in \Theta_0, \quad j = \overline{0, m_{\theta}-1},
\end{equation}
for some $w \in \mathcal H$. The system $\mathscr H_1(\Lambda)$ is complete in $\mathcal H$ if and only if
the relations \eqref{scalw} imply $w = 0$. Let $w(t) = [\overline{w_1(t)}, \overline{w_2(t)}]^T$. Then \eqref{scalw} can
be rewritten in the form
\begin{gather} \nonumber
\int_{-\pi}^{\pi} w_1(t) \, dt = 0, \quad \theta = 0, \\ \label{smB1}
\left( \frac{d^j}{d\la^j} \int_{-\pi}^{\pi} \biggl( \frac{A_1(\la)}{B_1(\la)} w_1(t) + \la w_2(t) \biggr) \exp(i\la t) \, dt \right) \Big|_{\la = \la_{\theta}} = 0,
\quad \theta \in \Theta_1, \\  \label{smB2}
\left( \frac{d^j}{d\la^j} \int_{-\pi}^{\pi} \biggl( w_1(t) + \frac{ \la B_1(\la)}{A_1(\la)} w_2(t) \biggr) \exp(i\la t) \, dt \right)\Big|_{\la = \la_{\theta}} = 0,
\quad \theta \in \Theta_2,
\end{gather}
for all $\theta \in \Theta_0$, $j = \overline{0, m_{\theta}-1}$. In view of \eqref{passABS1} and \eqref{passABS2}, the relations 
\eqref{smB1} and \eqref{smB2} are equivalent to the following ones:
\begin{gather*}
\left( \frac{d^j}{d\la^j} \int_{-\pi}^{\pi} \biggl( \frac{S_1'(\pi,\la)}{S_1(\pi,\la)} w_1(t) - \la w_2(t) \biggr) \exp(i\la t) \, dt \right) \Big|_{\la = \la_{\theta}} = 0,
\quad \theta \in \Theta_1, \\  
\left( \frac{d^j}{d\la^j} \int_{-\pi}^{\pi} \biggl( w_1(t) - \frac{ \la S_1(\pi, \la)}{S_1'(\pi, \la)} w_2(t) \biggr) \exp(i\la t) \, dt \right) \Big|_{\la = \la_{\theta}} = 0,
\quad \theta \in \Theta_2.
\end{gather*}

Recall that $S_1(\pi, \la_{\theta}) \ne 0$ for $\theta \in \Theta_1$
and $S_1'(\pi, \la_{\theta}) \ne 0$ for $\theta \in \Theta_2$.
Consequently,   the relations \eqref{scalw} are equivalent to the relations
\begin{equation} \label{scalv}
   (w, v_{\theta j})_{\mathcal H} = 0, \quad \theta \in \Theta_0, \quad j = \overline{0, m_{\theta}-1},
\end{equation}
that imply $w = 0$ if and only if the system $\mathscr V(\Lambda)$ is complete.
\end{proof}

\begin{lem} \label{lem:complete}
Denote by $\Lambda$ a subspectrum  of the boundary value problem $L$  which  satisfies  \eqref{defLa} and condition $(B)$. 
Then the corresponding system $\mathscr H_1(\Lambda)$ is complete in $\mathcal H$. 
\end{lem}

\begin{proof}
In view of Lemma~\ref{lem:HV} and the assumptions of this lemma, it is sufficient to show 
that the system $\mathscr V(\Lambda)$ is complete. Suppose that the relations \eqref{scalv} hold for some $w \in \mathcal H$,
$w = [w_1, w_2]^T$. Then the entire function
\begin{equation} \label{defW}
    W(\la) = \int_{-\pi}^{\pi} (w_1(t) S_1'(\pi, \la) - w_2(t) \la S_1(\pi, \la)) \exp(i \la t) \, dt 
\end{equation}
has the zeros $\{ \la_{\theta} \}_{\theta \in \Theta_0}$, counting with multiplicities. Taking \eqref{intS} and \eqref{intSp} into account,
we obtain the estimate
\begin{equation} \label{estW}
    W(\la) = O(\exp(2 \pi |\mbox{Im}\, \la|)), \quad |\la| \to \iy.
\end{equation}
Moreover,
\begin{equation} \label{estW1}
    W(\la) = o(1), \quad \la \in \mathbb R, \quad \la \to +\iy.
\end{equation}

Since   $\Lambda$ has the form \eqref{defLa}, we can 
construct the following function
\begin{gather} \label{defD}
    D(\la) = \prod_{k = 1}^4 D_k(\la), \\ \nonumber
    D_k(\la) = \frac{\pi}{2} (\la - \la_{0k}) \prod_{\substack{n = -\iy \\ n \ne 0}}^{\iy} 
   \frac{\la_{nk} - \la}{2 n} \exp \left( \frac{\la - \be_k}{2 n}\right), \quad \la_{01} := 0.
\end{gather}

Relying on \cite[Lemma~2]{BB17} and \cite[Appendix B]{Bond18}, one can show that
\begin{equation} \label{reprDk}
   D_k(\la) = \sin \left( (\la - \be_k) \tfrac{\pi}{2} \right) + \varkappa_k(\la), \quad k = \overline{1, 4},
\end{equation}
where $\varkappa_k(\la)$  are entire functions of exponential type not greater than $\tfrac{\pi}{2}$ for  $k = \overline{1, 4}.$ Consequently, the following estimate holds
\begin{equation} \label{estD}
    |D(\la)| \ge C \exp(2 \pi |\mbox{Im}\,\la|), \quad \eps < |\arg \la| < \pi-\eps, \quad |\la| \ge \la^*,
\end{equation}
for some positive $\eps$ and $\la^*$.

Obviously, the function $\dfrac{W(\la)}{D(\la)}$ is entire.  The estimates \eqref{estW} and \eqref{estD} imply
that $\dfrac{W(\la)}{D(\la)}$ is bounded in the sectors $\{ \la \in \mathbb C \colon \eps < |\arg \la| < \pi-\eps\}.$  
Applying Phragmen-Lindel\"of's and Liouville's theorems \cite{BFY},   that $\dfrac{W(\la)}{D(\la)} \equiv C$.
Taking \eqref{estW1}, \eqref{defD} and \eqref{reprDk} into account, we  have $W(\la) \equiv 0$.

Without loss of generality, we assume that $S_1(\pi, 0) \ne 0$ and $S_1'(\pi, 0) \ne 0$. Those conditions can be achieved by a 
shift of the spectrum. Then it follows from \eqref{defW} and $W(\la) \equiv 0$, that the function
$$
   W_1(\la) := \int_{-\pi}^{\pi} w_1(t) \exp(i \la t) \, dt 
$$
has zeros at $\la = 0$ and all the zeros of $S_1(\pi, \la)$, counting with their multiplicities. Therefore
$\dfrac{W_1(\la)}{S_1(\pi, \la)}$ is an entire function. Applying  similar arguments as above, we can show  $W_1(\la) \equiv 0$
and consequently  $w_1 = 0$ in $L_2(-\pi, \pi).$  In view of \eqref{defW} and $W(\la) \equiv 0$, we also have $w_2 = 0$ in $L_2(-\pi, \pi)$.
Thus we have proved that the relations \eqref{scalv} imply $w = 0$, so the systems $\mathscr V(\Lambda)$ and
$\mathscr H_1(\Lambda)$ are complete in $\mathcal H$.
\end{proof}

\begin{proof}[Proof of Theorem~\ref{thm:uniq1}]
Suppose that the boundary value problems $L$ and $\tilde L$ together with their subspectra $\Lambda$ and $\tilde \Lambda$
satisfy the conditions of Theorem~\ref{thm:uniq1}. 

It follows from the proof of Lemma~\ref{lem:d1} that $\{ \be_k \}_{k =1}^{2m-1}$ are the roots
of the equation
\begin{equation} \label{eqal1}
	\sum_{j = 1}^{m-1} \cot (\la - \al_j) \pi = \frac{2(1 - \cos \la \pi)}{\sin \la \pi}.
\end{equation}
Thus, the number $(\al_1 \bmod 1)$ can be found from \eqref{eqal1} by substituting  $\la = \be_k$ 

Since $p_j = \tilde p_j$ and $q_j = \tilde q_j$ for $j = \overline{2, m}$, we have $S_j(x_j, \la) \equiv \tilde S_j(x_j, \la)$
for $j = \overline{2, m}$. Consequently, $\al_j = \tilde \al_j$ for $j = \overline{2, m}$. The equality $\Lambda = \tilde \Lambda$
together with \eqref{defLa} imply $\be_k = \tilde \be_k$, $k = \overline{1, 4}$. Therefore \eqref{eqal1} yields
$\al_1 \equiv \tilde \al_1 \pmod 1$. 

There are two possible cases: 

\smallskip

(i) $\al_1 \equiv \tilde \al_1 \pmod 2$,

\smallskip

(ii) $\al_1 \equiv \tilde \al_1 + 1 \pmod 2$.

\smallskip

In the case (i), the relation \eqref{defG1} implies $G_1(\la) \equiv \tilde G_1(\la)$. Together with \eqref{scal4}, 
this yields $\mathscr G_1(\Lambda) = \tilde {\mathscr G}_1(\tilde \Lambda)$. The completeness of the system 
$\mathscr H_1(\Lambda)$ follows from Lemma~\ref{lem:complete}, so we can proceed the proof of Theorem~\ref{thm:uniq1-abs}.

In the case (ii), we have $G_1(\la) \equiv -G_1(\la)$, $\mathscr G_1(\Lambda) = -\tilde {\mathscr G}_1(\tilde \Lambda)$
and $\mathscr H_1(\Lambda) = \tilde {\mathscr H}_1(\tilde \Lambda)$.
Using \eqref{main1} and the similar relation for $\tilde L$, we conclude that $f_1 = -\tilde f_1$, so
$K_1(t) = -\tilde K_1(t)$, $N_1(t) = -\tilde N_1(t)$. The relations \eqref{intS} and \eqref{intSp} imply
$S_1(\pi, \la) \equiv -\tilde S_1(\pi, \la)$, $S_1'(\pi, \la) \equiv -\tilde S_1'(\pi, \la)$. Consequently, the Weyl functions
coincide: $\dfrac{S_1'(\pi, \la)}{S_1(\pi, \la)} \equiv \dfrac{\tilde S_1'(\pi, \la)}{\tilde S_1(\pi, \la)}$.
Hence $p_1 = \tilde p_1$, $q_1 = \tilde q_1$ and $\al_1 = \tilde \al_1$, that leads to the contradiction, so the case (ii)
is impossible.
\end{proof}

\begin{remark}
It follows from the proof of Theorem~\ref{thm:uniq1}, that the condition $\al_1 = \tilde \al_1$ in the statement of 
Theorem~\ref{thm:uniq1-abs} can be changed by $\al_1 \equiv \tilde \al_1 \pmod 1$. 
\end{remark}

\bigskip

{\large \bf 5. Constructive Solution of Inverse Problem~\ref{ip:1}}

\bigskip

In this section, we obtain a constructive procedure for solving Inverse Problem~\ref{ip:1} of recovering the pencil coefficients
on the boundary edge. The key role in our algorithm will be played by the relations \eqref{main1}. First we will prove the following theorem.

\begin{thm} \label{thm:Riesz}
Let $\Lambda$ be a subspectrum of the problem $L$ which satisfies  \eqref{defLa} and assumption $(B).$  Then the system 
of vector-functions $\mathscr H_1(\Lambda)$ defined in \eqref{defHG1}  is an unconditional basis in $\mathcal H$,
i.e. the normalized system $\{ h_{1,\theta j}(t) / \| h_{1,\theta j}(t) \|_{\mathcal H} \}$ is a Riesz basis.
\end{thm}

\begin{proof}
In order to show that a system is a Riesz basis in a Hilbert space, it is sufficient to prove, that this system is complete
and quadratically close to another Riesz basis (see \cite[Section~1.8.5]{FY01}). 
 The assumptions of this theorem and Lemma~\ref{lem:complete}  imply 
the system $\mathscr H_1(\Lambda)$ is complete. 

Since $\Lambda$ satisfies \eqref{defLa}, where the numbers $\{ \be_k \}_{k = 1}^4$ are distinct,
there may be only a finite number of multiple values in $\Lambda$ in view of the asymptotics~\eqref{asymptla}. 
Therefore for $\theta = (n, k) \in \Theta$ with sufficiently large $|n|$, we have $m_{\theta} = m_{nk} = 1$.
By virtue of \eqref{passABS1}, \eqref{passABS2} and \eqref{defV}, we have
$h_{1,\theta 0}(t) = c_{\theta} v_{\theta 0}(t)$ for $\theta \in \Theta$ and some nonzero constants $c_{\theta}$.
Consequently, in order to prove the claim of the theorem, it remains to show that the system $\mathscr V(\Lambda)$ 
is quadratically close to a Riesz basis. 

In view of \eqref{defV}, we have
$$
    v_{n k 0}(t) = \begin{bmatrix}  S_1'(\pi, \la_{nk}) \\ -\la_{nk} S_1'(\pi, \la_{nk}) \end{bmatrix} \exp(i \la_{nk} t), 
$$

The asymptotic relations for $\la_{nk}$, $S_1(\pi, \la)$ and $S_1'(\pi, \la)$ (formulas \eqref{asymptla}, \eqref{asymptS} and \eqref{asymptSp}) yield
$$
    v_{n k 0}(t) = v_{nk}^0(t) + O\left( |n|^{-1} \right), \quad  |n| \to \iy, \quad k = \overline{1, 4},
$$
where 
\begin{equation} \label{defv0}
    v_{nk}^0(t) = \begin{bmatrix} 
			\cos (\be_k - \al_1) \pi \\
			-\sin (\be_k - \al_1) \pi
		  \end{bmatrix} \exp(i (2 n + \be_k) t).				
\end{equation}
Thus, the system $\mathscr V(\Lambda)$ is quadratically close to the system 
$\mathscr V^0 := \{ v_{nk}^0(t) \}_{n \in \mathbb Z, \, k = \overline{1, 4}}$
in $\mathcal H$:
$$
     \sum_{\substack{n \in \mathbb Z\\ |n| \ge n^*}} \sum_{k = 1}^4 \| v_{nk0} - v_{nk}^0 \|^2_{\mathcal H} < \iy.			
$$

The following Lemma~\ref{lem:V0} finishes the proof of the theorem.
\end{proof}

\begin{lem} \label{lem:V0}
The system of vector-functions $\mathscr V^0 := \{ v_{nk}^0(t) \}_{n \in \mathbb Z, \, k = \overline{1, 4}}$
is a Riesz basis in $\mathcal H$, where $v_{nk}^0(t)$ were defined in \eqref{defv0}, $\{ \be_k \}_{k =1}^4$ are distinct real numbers
from $(-1, 1)$, $\al_1 \in \mathbb R$.
\end{lem}

\textit{Proof.}
The completeness of the system $\mathscr V^0$ can be easily shown by the standard methods (see, for example, the proof of Lemma~\ref{lem:complete}).
Denote $\mathcal J := \{ (n, k) \colon n \in \mathbb Z, \, k = \overline{1, 4} \}$.
Let $\{ c_{nj} \}_{(n, j) \in \mathcal J}$ be an arbitrary complex sequence in  $l_2$, and let
\begin{equation} \label{defv}
    v(t) = \sum_{(n, j) \in \mathcal J} c_{nj} v_{nj}^0(t). 
\end{equation}
We will prove that there exist positive constants $M_1$ and $M_2$, such that 
$$
    M_1 \sum_{(n, j) \in \mathcal J} |c_{nj}|^2 \le \| v \|_{\mathcal H}^2 \le M_2 \sum_{(n, j) \in \mathcal J} |c_{nj}|^2.
$$ 
This implies  $\mathscr V^0$ is a Riesz basis.

Indeed, using \eqref{defv} and  \eqref{defv0}, one can easily get 
\begin{equation} \label{quad}
    \| v \|_{\mathcal H}^2 = \sum_{(n, j) \in \mathcal J} \sum_{(k, l) \in \mathcal J} \overline{c_{nj}} c_{kl} (v_{nj}^0, v_{kl}^0)_{\mathcal H}
\end{equation}
and 
\begin{gather*}
    (v_{nj}^0, v_{kl}^0)_{\mathcal H} = \frac{\sin 2 (\be_l - \be_j) \pi}{2k - 2n + \be_l - \be_j}, \quad (n, j), (k, l) \in \mathcal J, \\
    (v_{nj}^0, v_{kl}^0)_{\mathcal H} = 2 \pi, \quad n = k, \: j = l.
\end{gather*}
Note that 
$$
    \int_{-\pi}^{\pi} \exp(-i(4n + 2\be_j)t) \exp(i(4k + 2\be_l)t) \, dt =  \frac{\sin 2 (\be_l - \be_j) \pi}{2k - 2n + \be_l - \be_j},
     \quad (n, j), (k, l) \in \mathcal J,
$$
i.e. the coefficients $(v_{nj}^0, v_{kl}^0)_{\mathcal H}$ of the quadratic form \eqref{quad} equal to
$(e_{nj}, e_{kl})_{L_2(-\pi, \pi)}$, where $e_{nj} := \exp(i(4 n + 2 \be_j) t)$, $(n, j) \in \mathcal J$.
Thus, the system $\mathscr V^0$ is a Riesz basis in $\mathcal H$ if and only if the system 
$\mathscr E := \{ e_{nj} \}_{(n, j) \in \mathcal J}$ is a Riesz basis in $L_2(-\pi, \pi)$. In order 
to investigate the Riesz-basicity of the exponent system,  we  use the following theorem.

\begin{thm}[Levin and Ljubarski\u{\i}\cite{LL75}] \label{thm:LL}
The system of functions $\{ \exp(i \la_k t) \}$, where $\{ \la_k \}$ is the set of zeros of a sine-type function $S(\la)$,
is a Riesz basis in $L_2(-\pi, \pi)$. An entire function of exponential type $S(\la)$ is called the sine-type function,
if for some constants $c$, $C$ and $K$, depending on $S(\la)$, the following inequality is valid:
$$
   0 < c < |S(\la)| \exp(-\pi|\mbox{Im}\,\la|) < C < \iy, \quad |\mbox{Im}\, \la| > K,
$$
and $\inf\limits_{k\ne j} |\la_k - \la_j| > 0$.
\end{thm}

Clearly, the numbers $(4 n + 2 \be_j)$, $(n, j) \in \mathcal J$, are the zeros of the sine-type function
$$
    S(\la) := \prod_{j = 1}^4 \sin \left( (\la - \be_j) \tfrac{\pi}{4} \right),
$$
so by virtue of Theorem~\ref{thm:LL}, the system $\mathscr E$ is a Riesz basis. Hence $\mathscr V^0$ is also a Riesz basis. 
$\hfill\Box$ 

\medskip

In view of Theorem~\ref{thm:Riesz}, if a subspectrum $\Lambda$ satisfies \eqref{defLa} and assumption $(B)$,
one can solve Inverse Problem~\ref{ip:1} by using the following algorithm.

\begin{alg}
Suppose functions $\{ p_j \}_{j = 2}^m$, $\{ q_j \}_{j = 2}^m$  and a subspectrum $\Lambda$, which satisfies  \eqref{defLa} and assumption ($B$), are given. 

\begin{enumerate}
\item Find the solutions $S_j(x_j, \la)$ of equations~\eqref{eqv} for $j = \overline{2, m}.$ 
\item Find 
$$
\al_j = \frac{1}{\pi} \int_0^{\pi} p_j(t) \, dt, \: j = \overline{2, m}, \quad \be_1 = \lim_{n \to \iy} (\la_{n1} - 2n),
$$
\item Solving eq.~\eqref{eqal1} for $\la = \be_1$, find 
$$
   \al_1 = \be_1 - \frac{1}{\pi} \mbox{arccot}\, \left( \frac{2(1 - \cos \be_1 \pi)}{\sin \be_1 \pi} - \sum_{j = 2}^{m-1} \cot (\be_1 - \al_j) \pi\right).
$$
\item Construct $p_1$ and $q_1$, following the steps~2--6 of Algorithm~\ref{alg:1}.
\end{enumerate} 
\end{alg}

Note that the number $\al_1$ can be determined incorrectly on step~3. In fact, we find ($\al_1 \bmod 1$). Nevertheless, 
in view of the proof of Theorem~\ref{thm:uniq1}, it does not influence the final result.

\bigskip

{\large \bf 6. Uniqueness Theorem for Inverse Problem~\ref{ip:m}}

\bigskip

The main goal of this section  is to prove the uniqueness Theorem~\ref{thm:uniqm}.
Our arguments resemble the proof of Theorem~\ref{thm:uniq1} in Section~6, so we will demonstrate only the general scheme and
will not elaborate into details. We start with the proof of the auxiliary lemma.

\begin{proof}[Proof of Lemma~\ref{lem:om}]
%Let assumption $(C)$ hold, i.e. $\om_n \ne 0$ for all $n \in \mathbb Z$. 
We have to prove that 
$d_m(\nu_n) \ne 0$ for all  zeros $\{ \nu_n \}_{n \in \mathbb Z}$  of $S_m(\pi, \la).$
Suppose on the contrary, $d_m(\nu_n) = 0$ for some $n \in \mathbb Z$. That means
\begin{equation} \label{CS1}
    S_m'(\pi, \nu_n) + C_m(\pi, \nu_n) - 2 = 0.
\end{equation}

Using \eqref{eqv} and \eqref{ic} for $j = m$, one can easily derive the relation
$$
   C_m(\pi, \la) S_m'(\pi, \la) - C_m'(\pi, \la) S_m(\pi, \la) \equiv 1.
$$
For $\la = \nu_n$, that yields 
\begin{equation} \label{CS2}
    C_m(\pi, \nu_n) S'_m(\pi, \nu_n) = 1.
\end{equation}

The relations \eqref{CS1} and \eqref{CS2} together yield $C_m(\pi, \nu_n) = S_m'(\pi, \nu_n) = 1$. Then   
$\omega_n=Q(\nu_n) = C_m(\pi, \nu_n) - S_m'(\pi, \nu_n)=0$.  This is a contradiction to the assumption ($C$). \end{proof}

Now let us derive the main equations for Inverse Problem~\ref{ip:m}.
Rewrite the relation \eqref{defDelta} in the form
\begin{equation} \label{Deltam}
   \Delta(\la) = A_m(\la) S_m(\pi, \la) + B_m(\la) d_m(\la), 
\end{equation}
where
$$
    A_m(\la) := \sum_{j = 1}^{m-1} S_j'(\pi, \la) \prod_{\substack{k = 1 \\ k \ne j}}^{m-1} S_k(\pi,\la), \quad
    B_m(\la) :=  \prod_{k = 1}^{m-1} S_k(\pi, \la).
$$

For simplicity, we assume that $\al_m = 0$.
Using \eqref{intSp} and \eqref{intC} for $j = m$, we conclude that
\begin{equation} \label{intdm}
    d_m(\la) = 2 \cos \la \pi - 2 + \int_{-\pi}^{\pi} T_m(t) \exp(i \la t) \, dt,
\end{equation}
where $T_m(t) = N_m(t) + L_m(t)$ is  continuous on $[-\pi, \pi]$.
Using \eqref{intSp}, \eqref{Deltam} and \eqref{intdm}, we obtain the relation
\begin{equation} \label{intABm}
   A_m(\la) \int_{-\pi}^{\pi} K_m(t) \exp(i \la t) \, dt + \la B_m(\la) \int_{-\pi}^{\pi} T_m(\la) \exp(i \la t) \, dt 
- G_m(\la) = \la \Delta(\la),
\end{equation}
where
$$
   G_m(\la) = -A_m(\la) \sin \la \pi - \la B_m(\la) (2 \cos \la \pi - 2).
$$
The index $m$ denotes the connection with the inverse problem, that consists in recovering the pencil coefficients
on the edge $e_m$, i.e. on the loop.

Define the vector functions
$$
    f_m(t) = \begin{bmatrix}
                \overline{K_m(t)} \\ \overline{T_m(t)}
	     \end{bmatrix},
    \quad
    h_m(t, \la) = \begin{bmatrix}
                     A_m(\la) \\ \la B_m(\la)
		\end{bmatrix} \exp(i \la t).		
$$

Since the functions $K_m$ and $T_m$ are continuous on $[-\pi, \pi]$, we have $f_m \in \mathcal H$.
Clearly, the vector function $h_m(t, \la)$ and its derivatives with respect to $\la$ also belong
to $\mathcal H$ for each fixed $\la \in \mathbb C$. Therefore the relation \eqref{intABm} can be rewritten
in the form
\begin{equation} \label{scalm}
  (f_m(t), h_m(t, \la))_{\mathcal H} = G_m(\la) + \la \Delta(\la).
\end{equation}

Suppose assumption $(C)$ is fulfilled for the problem $L$.
Let $\Lambda = \{ \lambda_{\theta} \}_{\theta \in \Theta}$ be a subspectrum $L$ which satisfies assumption $(D)$. 
Recall that $\la_{\theta}$ is a zero of the function $\Delta(\la)$ of multiplicity at least $m_{\theta}$.
For $\theta \in \Theta_1$ we have $S_m(\pi, \la_{\theta}) \ne 0$
and $B_m(\la_{\theta}) \ne 0$. 
For $\theta \in \Theta_2$ we have $S_m(\pi, \la_{\theta}) = 0$ and $B_m(\la_{\theta}) = 0$,
but in view of $(C)$, $d_m(\la_{\theta}) \ne 0$ and $A_m(\la_{\theta}) \ne 0$. For simplicity, we also assume that 
$0 \not \in \Lambda$, $0 \not \in \Theta$, 
and $\Theta_0 := \Theta \cup \{ 0 \}$, $\la_0 := 0$, $m_0 := 1$. 
The relation \eqref{intS} for $j = m$ and the equality $\al_m = 0$ yield
$$
\int_{-\pi}^{\pi} K_m(t) \, dt = 0.
$$

Thus, according to the mentioned facts, we arrive at the following main equations
\begin{equation} \label{mainm}
(f_m, h_{m, \theta j}(t, \mu))_{\mathcal H} = G_{m, \theta j}, \quad \theta \in \Theta_0, \quad j = \overline{0, m_{\theta}-1},
\end{equation}
where
\begin{gather*}
    h_{m,00}(t) = \begin{bmatrix} 1 \\ 0 \end{bmatrix}, \quad G_{m, 00} = 0, \\
    h_{m,\theta j}(t) = \frac{d^j}{d\la^j} \left(\frac{h_m(t, \la)}{B_m(\la)}\right) \Big|_{\la = \la_{\theta}}, \quad
    G_{m,\theta j} = \frac{d^j}{d\la^j} \left(\frac{G_m(\la)}{B_m(\la)}\right) \Big|_{\la = \la_{\theta}}, \quad
    \theta \in \Theta_1, \\
    h_{m,\theta j}(t) = \frac{d^j}{d\la^j} \left(\frac{h_m(t, \la)}{A_m(\la)}\right) \Big|_{\la = \la_{\theta}}, \quad
    G_{m, \theta j}(\mu) = \frac{d^j}{d\la^j} \left(\frac{G_m(\la)}{A_m(\la)}\right) \Big|_{\la = \la_{\theta}}, \quad
    \theta \in \Theta_2.	
\end{gather*}

Define the sequences
\begin{equation} \label{defHGm}
   \mathscr H_m(\Lambda) := \{ h_{m, \theta j}(t) \}_{\theta \in \Theta_0, \, j = \overline{0, m_{\theta} - 1}}, \quad
   \mathscr G_m(\Lambda) := \{ G_{m, \theta j}\}_{\theta \in \Theta_0, \, j = \overline{0, m_{\theta} - 1}}.
\end{equation}

The following theorem is analogous to Theorem~\ref{thm:uniq1-abs}.

\begin{thm} \label{thm:uniqm-abs}
Suppose that $p_j = \tilde p_j$ in $AC[0,\pi]$  and $q_j = \tilde q_j$ 
in $L(0, \pi)$ for $j = \overline{1, m-1},$ 
 $\Lambda = \tilde \Lambda$, $\Omega = \tilde \Omega$, $\al_m = \tilde \al_m = 0$,  and assumptions
$(C)$ and $(D)$ hold for the both pairs $(L, \Lambda)$ and $(\tilde L, \tilde \Lambda)$,
and the system $\mathscr H_m(\Lambda)$ is complete in $\mathcal H$. Then $p_m = \tilde p_m$ in $AC[0, \pi]$
and $q_m = \tilde q_m$ in $L(0, \pi)$. Thus, under the above assumptions, the solution of Inverse Problem~\ref{ip:m}
is unique.
\end{thm}

\begin{proof}
Let the problems $L$, $\tilde L$ and their subspectra $\Lambda$, $\tilde \Lambda$ satisfy the conditions of the theorem.
Analogously to the proof of Theorem~\ref{thm:uniq1-abs}, we can show that $S_m(\pi, \la) \equiv \tilde S_m(\pi, \la)$,
$d_m(\pi, \la) \equiv \tilde d_m(\pi, \la)$. Since assumption $(C)$ holds and $\Omega = \tilde \Omega$,
we apply Theorem~\ref{thm:period} and conclude that $p_m = \tilde p_m$ in $AC[0, \pi]$ and $q_m = \tilde q_m$ 
in $L(0, \pi)$.
\end{proof}

\begin{proof}[Proof of Theorem~\ref{thm:uniqm}]
Let the problems $L$, $\tilde L$ and their subspectra $\Lambda$, $\tilde \Lambda$ satisfy the conditions of Theorem~\ref{thm:uniqm}.
In particular, the subspectrum $\Lambda$ has the form~\eqref{defLa}, and asymptotic formulas \eqref{asymptla}
are valid for $\la_{nk}$. Similarly to the proof of Lemma~\ref{lem:complete}, one can show that, under those conditions, 
the system of vector-functions $\mathscr H_m(\Lambda)$ is complete in $\mathcal H$. Consequently, Theorem~\ref{thm:uniqm-abs} leads to 
the assertion of Theorem~\ref{thm:uniqm} and this completes the proof.

\end{proof}

Thus, we have shown the uniqueness for solution of Inverse Problem~\ref{ip:m}. If the system $\mathscr H_m(\Lambda)$
is an unconditional basis in $\mathcal H$, one can obtain a constructive algorithm for solving this problem,
similar to Algorithm~\ref{alg:1}.

\medskip

{\bf Acknowledgments}. This work was supported by the Mathematics Research Promotion Center of Taiwan.
The author N.P.~Bondarenko was also supported by Grants
20-31-70005 and 19-01-00102 of the Russian Foundation for Basic Research. 
The author Chung-Tsun Shieh was partially supported by Ministry of Science and Technology, Taiwan under 
Grant 106-2115-M-032 -004 -.

\medskip

\noindent Natalia Pavlovna Bondarenko \\
1. Department of Applied Mathematics and Physics, Samara National Research University, \\
Moskovskoye Shosse 34, Samara 443086, Russia, \\
2. Department of Mechanics and Mathematics, Saratov State University, \\
Astrakhanskaya 83, Saratov 410012, Russia, \\
e-mail: {\it BondarenkoNP@info.sgu.ru}

\medskip

\noindent Chung-Tsun Shieh \\
Department of Mathematics, 
Tamkang University, \\
151 Ying-chuan Road Tamsui,  New Taipei County, Taiwan. \\
e-mail: {\it ctshieh@mail.tku.edu.tw}

\end{document}